\documentclass[10pt,reqno,twoside,notitlepage,12pt]{amsart}%
\usepackage{amsmath}
\usepackage{amsfonts}
\usepackage{amssymb}
\usepackage{graphicx}%
\providecommand{\U}[1]{\protect\rule{.1in}{.1in}}
\newtheorem{theorem}{Theorem}

\newtheorem{corollary}[theorem]{Corollary}

\newtheorem{definition}[theorem]{Definition}
\newtheorem{example}[theorem]{Example}
\newtheorem{examples}[theorem]{Examples}

\newtheorem{proposition}[theorem]{Proposition}
\newtheorem{remark}[theorem]{Remark}

\newenvironment{Exmps}{\begin{examples}\em}{\end{examples}}
\newtheorem{prf}{Proof}

\begin{document}

\title{`Ultrafilters Relaxed' -- Species of Families of Subsets}
\author{Eliahu Levy}
\address{Eliahu Levy, Department of Mathematics The Technion, Israel Institute of
Technology, Technion City, Haifa 32000, Israel.}
\email{eliahu@math.technion.ac.il.}


\date{}


\begin{abstract}
An assortments of `species' of families of subsets of a set and some of their properties are investigated, with an eye on the logic and limit role they may play as `relaxed parallels' to ultrafilters.
\end{abstract}

\maketitle

\tableofcontents

\section{Introduction, Points vs.\ Families\label{s:PvF}}
Let $X$ be a set. For points $x\in X$ and subsets $S\subset X$, there is the \textit{truth-value} whether $X\in S$ or not.

But with \textit{families of subsets} $\mathcal{F}\subset\mathcal{P}(X)$, where $\mathcal{P}(x)$ is the set of subsets of $X$, there is also a truth-value -- whether $S\in\mathcal{F}$ or not.

Indeed, for every $x\in X$, the family
$$\mathcal{U}_x:=\{S\subset X\,|\,x\in S\}$$
`does the same as $x$' for that matter: by definition, the truth-value whether $x\in S$ is the same as the truth-value whether $S\in\mathcal{U}_x$.

In this sense, judging by these truth-values, the map (obviously an injection) $x\mapsto\mathcal{U}_x$ \textit{embeds $X$ into $\mathcal{P}(X)$}, making $\mathcal{P}(X)$ a kind of extension of $X$ -- `psudo-points' defined by to which subsets they belong, in a way other than all $x\in X$.

Indeed, $\mathcal{U}_x$ in an instance of an \textit{ultrafilter} in $X$. and using the Axiom of Choice one proves that in an infinite set there are many other ultrafilters.

To see what we aim at, note that any subset of $X$ can be viewed as a \textit{property of elements of $X$}, namely the property for $x$ to be in $S$, thus for $S$ to be the set of elements satisfying the property.

Logical connectives among properties will make operations on the sets:

\medskip

\textbf{I}. \textbf{NOT} -- giving the \textit{complement} $S^c$, defined as $S^c:=\{x\in X\,|\,x\notin S\}$. Indeed, to say that $x\in S^c$ is to say that NOT $x\in S$

\textbf{II}. \textbf{AND} -- giving the \textit{intersection} $S\cap T$ of two subsets. Indeed, to say that $x\in S$ AND $x\in T$ is to say that $x\in S\cap T$

\textbf{III}. \textbf{OR} -- giving the \textit{union} $S\cup T$. To say that $x\in S$ OR $x\in T$ is to say that $x\in S\cup T$.

\textbf{IV} \textbf{IMPLICATION} -- giving \textit{inclusion} of subsets. Indeed, to say that $x\in S$ always implies $x\in T$,
is to say that $S\subset T$.

\textbf{V} \textbf{TRUE} -- giving $S$ as the whole $X$ for which $x\in X$ is always true.

\textbf{VI} \textbf{FALSE} -- giving $S$ as the empty set $\emptyset$ for which $x\in\emptyset$ is always false.

\medskip

But given a family of subsets $\mathcal{F}\subset\mathcal{P}(S)$. \textit{Can we plug $\mathcal{F}$ instead of $x$}
in \textbf{I - VI}, i.e.,

\medskip

\textbf{I} Does $S^c\in\mathcal{F}$ hold if and only if \textbf{NOT} $S\in\mathcal{F}$?

\textbf{II} Does $S\cap T\in\mathcal{F}$ hold if and only if $S\in\mathcal{F}$ \textbf{AND} $S\in\mathcal{F}$?

\textbf{III} Does $S\cup T\in\mathcal{F}$ hold if and only if $S\in\mathcal{F}$ \textbf{OR} $S\in\mathcal{F}$?
$S\in\mathcal{F}$?

\textbf{IV} Does $S\subset T$ make $S\in\mathcal{F}$ always imply $T\in\mathcal{F}$?

\textbf{V} Are we guarantied that $X$ belongs to $\mathcal{F}$?

\textbf{VI} Are we guarantied that $\emptyset$ does not belong to $\mathcal{F}$?

\medskip

Certainly these may not hold. Clearly they hold for the $\mathcal{U}_x$.

An \textbf{ultrafilter} is characterized by \textit{all \textbf{I - IV} holding}.

This means that for ultrafilters first-order logical operations carry over.

And since in an infinite set we will have them besides the $\mathcal{U}_x$, one applies that to make, say, from the natural numbers $\mathbb{N}$ non-standard models of its first-order theory, i.e.\ genuine extensions of $\mathbb{N}$ where first-order statements hold if and only if they hold in $\mathbb{N}$ (alas, the induction axiom, inasmuch as it refers to `any subset of $\mathbb{N}$' is not first-order -- that's much of the point).

But suppose we relax our requirements, assume that for $\mathcal{F}$ \textit{only some of \textbf{I -IV} hold}. In this note we try to remark on such families.

For instance, suppose only \textbf{IV} holds: if $S\subset T$ and $S\in\mathcal{F}$ then $T\in\mathcal{F}$. Then $\mathcal{F}$ is called an \textbf{eventual family} (see \cite{Censor-Levy} and \S\ref{sbs:Ev}).

Note that for eventual $\mathcal{F}$, the \textbf{only if} in \textbf{II} above and the \textbf{if} in \textbf{III} are automatic. \textbf{II} just says that $\mathcal{F}$ respects intersections -- if $S$ and $T$ belong to $\mathcal{F}$ also $S\cap T$ belongs -- an eventual family satisfying that is called a \textbf{filter}.

Suppose only \textbf{I} holds: $S^c\in\mathcal{F}$ if and only if $S\notin\mathcal{F}$. These will be referred to as \textbf{self-Aso} in \S\ref{s:Ev-Sa-Inn-Out}. That, of course, holds for ultrafilters, but not only for them. To take an example: let $X$ be finite with $2n+1$ elements and let $\mathcal{F}$ be the family of sets with $>n$ elements. This example is also eventual.

Recall that one easily proves that an ultrafilter in finite $X$ must be `fixed' -- an $\mathcal{U}_x$.

\subsection{Limits}

Families of subsets $\mathcal{F}$ just give to any $S\subset X$ a truth value, whether $S\in\mathcal{F}$ or not.

To speak `fancily': they extract from any $\{0,1\}$-valued function on $X$ a value in $\{0,1\}$.

The merit of that way of speaking is that, with respect to some $\mathcal{F}$, one often can do similar things with some space $Y$ instead of $\{0,1\}$: extract a \textit{limit} in $Y$ from $Y$-valued functions on $X$.

Let's see how that goes with $\mathcal{F}$ an \textit{ultrafilter}.

Let $Y$ be a bounded interval in $\mathbb{R}$: we want $\mathcal{F}$ to extract a limit from a bounded real-value function $f$ on $X$.

Now, partition $Y$ into a finite number of disjoint small intervals $Y_1,Y_2,\ldots,Y_k$.

Then $f^{-1}(Y_1),f^{-1}(Y_2),\ldots,f^{-1}(Y_k)$ will be a disjoint partition of $X$. Now, $\mathcal{F}$
cannot contain two of these, because then it would also contain their intersection -- empty. It cannot contain none of them, then it will not contain their union -- the whole $X$. Therefore it contains exactly one of them. So, modulo the ultrafilter $\mathcal{F}$, i.e.\ outside some subset not in the ultrafilter, $f$ takes values in some numerical set of
as small diameter as we wish.

Things that hold modulo $\mathcal{F}$ cannot be in contradiction, indeed any finite conjunction of them also hold modulo $\mathcal{F}$ since as an ultrafilter it contains the intersection of any two of its members.

So one concludes that for any such bounded real valued function $f$ on $X$, and an ultrafilter $\mathcal{F}$ there is a unique `limit' $a\in\mathbb{R}$ characterized by: for any neighborhood $U$ of $x$ in the reals, modulo the ultrafilter $f$ takes values in $U$.

And one easily sees that this limit depends linearly on $f$, nay, for any continuous real function $\phi$ of $n$ variables, the limit of $\phi$ applied to $n$ functions equals $\phi$ applied to their limits!

But these notions can be applied to much more general families $\mathcal{F}$.

Suppose we retain the definition of a limit point of a $Y$-valued function on $X$, $Y$ a Haudorff topological space, as any point $y\in Y$ such that

\medskip

For any neighborhood $U$ of $y$ in $Y$, any set containing $\{x\in X\,|\,f(x)\in U\}$ belongs to $\mathcal{F}$.

Thus that limit notion will be the same for $\mathcal{F}$ and for \textit{the eventual core of $\mathcal{F}$}, defined as the maximal eventual family contained in $\mathcal{F}$, given by
       $$\mathcal{F}':=\{S\subset X\,|\,\text{any\,}\,S'\supset S\,\text{ is in F.}\}$$
So we here assume from the start \textit{$\mathcal{F}$ eventual}.
\medskip

Then we have
\begin{itemize}
\item For $\mathcal{F}$ an \textit{ultrafilter} and $Y$ a bounded real interval, we saw that a unique limit exists always, whatever the function $f:X\to Y$. That holds for any compact $Y$.
\item Clearly, with $\mathcal{F}$ a filter a limit is always unique (yet need not exist).
\end{itemize}

\subsection{Focusing on Eventual Families}\label{sbs:Ev}
This section includes some points mentioned in \cite{Censor-Levy}, recapitulated here
for the sake of completeness.

Above the notion of eventual families of subsets was introduced, namely,

\begin{definition}
\label{def:event-co-event}Let $X$ be a set and let $\mathcal{F}$ be a family
of subsets of $X.$ The family $\mathcal{F}$ is called an \textbf{eventual family} if it is \textit{upper hereditary
with respect to inclusion}, i.e.,\ if%
\begin{equation}
S\in\mathcal{F},\,S^{\prime}\supseteq S\,\Rightarrow S^{\prime}\in\mathcal{F}.
\end{equation}
The family $\mathcal{F}$ is called a \textbf{co-eventual family} if it is \textit{lower hereditary with respect to
inclusion}, i.e.,\ if%
\begin{equation}
S\in\mathcal{F},\,S^{\prime}\subseteq S\Rightarrow S^{\prime}\in\mathcal{F}.
\end{equation}

\end{definition}

We mention in passing that Borg \cite{Borg-hereditary} uses the term
`hereditary family', in his work in the area of combinatorics, for exactly what we call here
`co-eventual family'.

Several simple observations regarding such families
can be made.

\begin{proposition}
\label{lem:simple-observ}(i) A family $\mathcal{F}$ of subsets of $X$ is
co-eventual iff its complement, i.e., the family of subsets of $X$ which are
not in $\mathcal{F}$, is eventual.

(ii) The empty family and the family of all subsets of $X$ are each both
eventual and co-eventual, and they are the only families with this property.
\end{proposition}

\begin{proof}
(i) This follows from the definitions. (ii) That the empty family and the
family of all subsets of $X$ are each both eventual and co-eventual is
trivially true. We show that if $\mathcal{F}$ is eventual and co-eventual and
is nonempty then it must contain all subsets of $X.$ Let $S\in\mathcal{F}$ and
distinguish between two cases. If $S=\emptyset$ then\thinspace\thinspace
$\mathcal{F}$ must contain all subsets of $X$ because $\mathcal{F}$ is
eventual. If $S\neq\emptyset$ let $x\in S$, then, since $\mathcal{F}$ is
co-eventual it must contain the singleton $\{x\}$. Consequently, the set
$\{x,y\},$ for any $y,$ is also in $\mathcal{F}$ and so $\{y\}\in\mathcal{F}$,
thus, all subsets of $X$ are contained in $\mathcal{F}$. Alternatively, if we
look at $S\in\mathcal{F}$, then for any subset $S^{\prime}$ of $X$,
$\mathcal{F}$ contains $S\cup S^{\prime}$ since $\mathcal{F}$ is eventual.
Then since $\mathcal{F}$ is co-eventual, it must contain $S^{\prime}$, leading
to the conclusion that it contains all subsets.
\end{proof}

\begin{remark}
\label{rem:filter}An eventual family $\mathcal{F}$ need not contain the
intersection of two of its members. If it does so for every two of its members
then it is a \textit{filter}.
\end{remark}

Similarly to the notion used in \cite{Danzig-Folkman-Shapiro} and
\cite{Lent-Censor} in the finite-dimensional space setting, we make here the
next definition.

\begin{definition}
\label{def:star-set}Given a family $\mathcal{F}$ of subsets of a set $X$, the
\textbf{star set associated with} $\mathcal{F}$, denoted by\thinspace\thinspace$\text{Star}%
(\mathcal{F}),$ is the subset of $X$ that consists of \textit{all\,}$x\in X$
\textit{such that the singletons\,}$\{x\}\in\mathcal{F}$, namely,%
\begin{equation}
\text{Star}(\mathcal{F}):=\{x\in X\mid\{x\}\in\mathcal{F}\}.
\end{equation}
\end{definition}

Suppose now that $X$ is a \textit{Hausdorff\,}Topological space.

\begin{definition}
\label{def:limit-pt-and-set}Let $\mathcal{F}$ be an eventual family of subsets
of $X$. A point $x\in X$ is called an \textbf{accumulation (or limit) point} of $\mathcal{F}$ if every (open)
neighborhood%
\footnote{Since, by definition, a neighborhood always contains an
\textit{open} neighborhood, considering all neighborhoods or just the open
ones does not make a difference here.}
of $x$ belongs to $\mathcal{F}$. The set of all accumulation points of $\mathcal{F}$ is called the
\textbf{limit set} of $\mathcal{F}$.
\end{definition}

\begin{proposition}
\label{prop:lim-set-closed}The limit set of an eventual family $\mathcal{F}$
is always closed.
\end{proposition}

\begin{proof}
We show that the complement of the limit set, i.e., the set of all
non-accumulation points, is open. The point $y$ is a non-accumulation point
iff it has an open neighborhood which does not belong to $\mathcal{F}$,
i.e.,\ when it is a member of some open set not in $\mathcal{F}$. Hence the
complement of the limit set is the union of all open sets not in $\mathcal{F}%
$, and by definition, in a topological space, the union of any family of open
sets is open.
\end{proof}

We turn our attention now to sequences in $X$, i.e.,\ maps $\mathbb{N}%
\rightarrow X,$ where $\mathbb{N}$ denotes the integers.

\begin{definition}
\label{def:push}Given are a family $\mathcal{F}$ of subsets of $X$\ and a
mapping between sets $f:X\rightarrow Y$. The family of subsets
of $Y$ whose inverse image sets $f^{-1}(S)$ belong to $\mathcal{F}$ will be
denoted by $\text{Push}(f,\mathcal{F)}$ and called the
\textbf{push} of $\mathcal{F}$ by $f$, namely,
\begin{equation}
\text{Push}(f,\mathcal{F)}:=\{S\subseteq Y\,\mid f^{-1}(S)\in\mathcal{F}\}.
\end{equation}
\end{definition}

Combining Definitions \ref{def:limit-pt-and-set} and \ref{def:push} the
following remark is obtained.

\begin{remark}
\label{Blaim:push}Let $\mathcal{E}$ be an eventual family of subsets
of $\mathbb{N}$ and let $f:\mathbb{N}\rightarrow X$ be defined by
some given sequence $(x_n)_{n\in\mathbb{N}}$ in $X$. The accumulation points
and the limit set of $(x_n)_{n\in\mathbb{N}}$ with respect to $\mathcal{E}$
are those defined with respect to the push of $\mathcal{E}$ by $f$ .
\end{remark}

The next examples emerge by using two different eventual families in
$\mathbb{N}$. The same `machinery' yields both `cases' via changing the
eventual family $\mathcal{E}$ in $\mathbb{N}$.

\begin{Exmps}
\label{exmps}

\begin{enumerate}
\item Take as $\mathcal{E}$ the family $\mathcal{H}$ of all subsets of $\mathbb{N}$
\textit{with finite complement}.

Then accumulation points/limits with respect to $\mathcal{H}$ are the usual limits,
and if there is a limit point then it is unique. This is the case, as one clearly sees,
in a Hausdorff space $X$ whenever $\mathcal{E}$ is a filter, as here $\mathcal{E}$ clearly is.

\item Now take as $\mathcal{E}$ the family $\mathcal{G}$ of all \textit{infinite} subsets of
$\mathbb{N}$.

Then being an accumulation point means \textit{being some accumulation point
of the sequence} in the usual sense, which in general, need not be unique.
Indeed, here $\mathcal{E}$ is not a filter.
\end{enumerate}
\end{Exmps}

When considering eventual families in $\mathbb{N}$ it is often desirable to
assume that they are \textit{finitely-insensitive}, as we define next. All our
examples have this property.

\begin{definition}
\label{def:finite-intensive}A family $\mathcal{E}$ of subsets of $\mathbb{N}$
is called a \textbf{finitely-insensitive family} if for any $S\in\mathcal{E}$, finitely changing $S$, which means here adding and/or deleting a finite number of its members, will result in a set $S^{\prime}\in\mathcal{E}$.
\end{definition}

\begin{definition}
\label{def:distill}Let $X$ be a Hausdorff topological space and let
$\mathcal{F}$ be an eventual family in $X$. The \textbf{closure} of an eventual family $\mathcal{F}$ in $X$,
denoted by $\text{cl\,}\mathcal{F}$, consists of all subsets
$S\subseteq X$ such that \textit{all the open subsets $U\subseteq X$
which contain $S$ belong to $\mathcal{F}$.}
\end{definition}

Clearly, $\mathcal{F}$ is always a subfamily of $\text{cl\,}\mathcal{F}$, and
the set of limit points of an eventual family $\mathcal{F}$, in a Hausdorff
topological space $X,$ is just $\text{Star\,}(\text{cl\,}\mathcal{F}),$ given
in Definition \ref{def:star-set}.

\section{Multi-sets and Multi-families\label{s:mult-fam}}
Here too some points mentioned in \cite{Censor-Levy} are recapitulated.

A \textbf{multi-set} (sometimes termed \textbf{bag}, or \textbf{mset}) is a
modification of the concept of a set that allows for multiple instances for
each of its elements. The number of instances given for each element is called
the multiplicity of that element in the multi-set. The multiplicities of
elements are any number in $\{0,1,\ldots,\infty\}$, see the corner-stone
review of Blizard \cite{Blizard-multiset-1989}.

Yet in contrast to \cite{Blizard-multiset-1989}, we shall not pursue the distinction, from our point
of view of purely philosophical motivation, between a multi-set and its `representing function' -- the $\{0,1,\ldots,\infty\}$-valued function giving the multiplicities. Thus,

\begin{definition}
(i) A \textbf{multi-set} $M$ in a set $X$ is a $\{0,1,\ldots,\infty\}$- valued function
$M$ on $X$, its value $M(x)$ at some $x\in X$ considered as the multiplicity of $x$ in $M$.
In particular $M(x)=0$ -- the multiplicity being $0$, means `$x$ not belonging'.

A subset $S\subseteq X$ is identified with a multi-set
which is the \textit{characteristical function} or \textit{indicator function} of $S,$ i.e.,%
\begin{equation}
\iota_{S}(x):=\left\{
\begin{array}
[c]{cc}%
1, & \text{if\,}x\in S,\\
0, & \text{if\,}x\notin S
\end{array}
\right.
\end{equation}

(ii) A \textbf{multi-family} $\mathcal{M}$ on a set $X$ is a multi-set in the
powerset $2^{X}$ of $X$ (i.e., all the subsets of $X$). For such subset $S$
$M(S)$ is the multiplicity of $S$ in $\mathcal{M}$.
A family $\mathcal{F}$ of subsets of $X$
is thus identified with the, here $\{0,1\}$-valued, multi-family on $X$ $\iota_{\mathcal{F}}$, the
\textit{characteristical function} or \textit{indicator function} of
$\mathcal{F}$.

\begin{equation}
\iota_{\mathcal{F}}(S):=\left\{
\begin{array}
[c]{cc}%
1, & \text{if\,}S\in\mathcal{F},\\
0, & \text{if\,}S\notin\mathcal{F}.
\end{array}
\right.
\end{equation}

(iii) A multi-family $\mathcal{M}$ on a set $X$ is called \textbf{increasing} if%
\begin{equation}
S,S^{\prime}\subseteq X,\,S\subseteq S^{\prime}\Rightarrow
\mathcal{M}(S)\leq\mathcal{M}(S^{\prime}),
\end{equation}
and called \textbf{decreasing} if%
\begin{equation}
S,S^{\prime}\subseteq X,\,S\subseteq S^{\prime}\Rightarrow\mathcal{M}(S)\geq\mathcal{M}(S^{\prime}).
\end{equation}
\end{definition}

Clearly, a family of subsets of $X$ is an eventual (resp.\ co-eventual) family if and only if the
multi-family that defines it is increasing (resp.\ decreasing).

The next example shows why these notions may be useful.

\begin{example}\label{ex:Gap}
Considering the set $\mathbb{N}$, for a, finite or infinite,
subset $S\subseteq\mathbb{N}$ write $S$ as%
\begin{equation}
S=\{n_1^S,n_2^S,\ldots\},
\end{equation}
where $n_{\ell}^S\in\mathbb{N}$ for all $\ell,$ and the sequence $(n_{\ell
}^S)_{\ell=1}^{L}$ (where $L$ is either finite or $\infty$) is strictly
increasing, i.e., $n_1^S<n_2^S<\ldots$. The \textbf{gaps} between consecutive elements in $S$ will
be the sequence of differences%
\begin{equation}
n_2^S-n_1^S-1,n_3^S-n_2^S-1,\ldots,
\end{equation}
where, if $S$ is finite add $\infty$ at the end. Defining%
\begin{equation}
\text{Gap }(S):=\limsup_{k}(n_{k+1}^S-n_{k}^S-1),
\end{equation}

makes $\text{Gap}$ a\textit{ }multi-family\textit{ }on\textit{ }$\mathbb{N}$,
thus taking values in $\{0,1,\ldots,\infty\}$, in particular, taking the value
$\infty$ for (among others) any finite $S$.

Note that if $\text{Gap }(S)$ is finite then there must be an infinite number
of differences $(n_{k+1}^S-n_{k}^S-1)$ equal to $\text{Gap }(S)$, but
this is not true for any larger integer - because by the definition of
$\limsup$ and because we are dealing with integer-valued items, a finite
$\limsup$ must actually be attained an infinite number of times.

Observe further that the larger the set $S$ is -- the smaller (or equal) is
$\text{Gap }(S)$. Thus, $\text{Gap}$ is a decreasing multi-family.

Define the complement-multi-family for some multi-family $\mathcal{G}$ on the
subsets of a set $X$ by%
\begin{equation}
(\text{co}\,\mathcal{G})(S):=\mathcal{G}(S^c),\quad\forall S\subseteq X
\label{eq:script-G-complement}%
\end{equation}

where $S^c$ is the complement of $S$ in $X$.

We will focus on $\text{coGap}:=\text{co}\,{Gap}$. For any $S\subseteq
\mathbb{N}$, let us denote by $c_{S}$ the maximal number of integers between
consecutive elements of $S,$ namely, between $n_{\ell}^S\in S$ and
$n_{\ell+1}^S\in S$. If $S$ has arbitrarily big such `intervals' between
consecutive elements then we write $c_{S}=\infty$. With this in mind,
$\text{coGap}=\text{Gap}^c$ is an increasing multi-famil\textit{y} equal
to $(c_{S})$\,$\forall S\subseteq\mathbb{N}$.
\end{example}

\subsection{Extensions of Notions Pertaining to Families to Multi-families\label{sbs:transferring}}

We now extend some of the notions of Subsection \ref{sbs:Ev} to multi-families.

\begin{definition}
\label{def:star-set copy(1)}Given a multi-family $\mathcal{M}$ on the subsets
of a set $X$.
The \textbf{star set associated with} $\mathcal{M}$, denoted by\thinspace\thinspace
$\text{Star }(\mathcal{M})$, is the multi-set $M$ on $X$ whose value on some $x$ -- the multiplicity of $x$
with respect to it, is defined to be the multiplicity $\mathcal{M}(\{x\})$ of the singleton $\{x\}$
with respect to $\mathcal{M}$.
\begin{equation}
\text{Star }(\mathcal{M})(x):=M(\{x\}).
\end{equation}
\end{definition}

\begin{definition}
\label{def:push copy(1)}Given are a multi-family $\mathcal{M}$ on the subsets of
$X$ and a mapping between sets $f:X\rightarrow Y$.
The \textbf{push} of $\mathcal{M}$ by $f$ is defined as the multi-family on the subsets
of $Y$ given by
\begin{equation}
\text{Push }(f,\mathcal{M})(S):=\mathcal{M}(f^{-1}(S)).
\end{equation}
\end{definition}

\begin{definition}
\label{def:finite-intensive copy(1)}A multi-family $\mathcal{M}$ of subsets of
$\mathbb{N}$
a \textbf{finitely-insensitive multi-family}
if for any $S\in\mathcal{M}$, finitely changing $S$, i.e.,\ adding and/or
deleting a finite number of its members, will not change its multiplicity,
i.e., will result in a set $S^{\prime}\in\mathcal{M}$ such that
$\mathcal{M}(S)=\mathcal{M}(S^{\prime})$.
\end{definition}

\begin{definition}
\label{def:distill copy(1)}Let $X$ be a Hausdorff topological space.
The \textbf{closure of an increasing multi-family }$M$\textbf{ in }$X$
, denoted by $\text{cl }\mathcal{M}$, is defined to be the (increasing)
multi-family given by
\begin{equation}
\text{cl }\mathcal{M}(S):=\min\{\mathcal{M}(U)\mid\text{
all\ \textit{open subsets }}U\subseteq X\text{ such that }S\subseteq
U\}.\text{ }%
\end{equation}

\end{definition}

\begin{definition}
\label{def:it:lim}Let $X$ be a Hausdorff topological space and let
$\mathcal{M}$ be an \textit{increasing} multi-family.
 The multi-set $M:=\text{Star}(\text{cl }\mathcal{M})$ will be called the
\textbf{multi-set-limit} of $\mathcal{M}$ and denoted by
$\lim\mathcal{M}$. It is thus given by, for any $x\in X$,
\begin{equation}
M(x)=\min\{\mathcal{M}(U)\mid\text{all open subsets }
U\subseteq X\text{ such that }x\in U\}.
\end{equation}
\end{definition}

Given a multi-family $\mathcal{M}$ on the subsets of $\mathbb{N}$.The `limiting notions'
with respect to $\mathcal{M}$ for a sequence $(x_n)_{n\in\mathbb{N}}$, are
defined as those with respect to $\text{Push }(f,\mathcal{M)}$ of
$\mathcal{M}$ to $X$ by the function $f:\mathbb{N}\rightarrow X$ which
represents the sequence $(x_n)_{\ }$. In particular, for an increasing
multi-family $\mathcal{M}$ on the subsets of $\mathbb{N}$, the multi-set limit of
$\text{Push }(f,\mathcal{M)}$ will be called the \textbf{multi-set-limit}
of $(x_n)$, denoted by $\lim_{\mathcal{M}}x_n.$

This multi-set $\mathcal{G}$ on $X$ can be described as follows. Given a point
$x\in X,$ consider the following subsets of $\mathbb{N}${%
\begin{equation}
S(U):=\{n\in\mathbb{N}\mid x_n\in U\},\text{ for open neighborhoods }U\text{ of }x.
\end{equation}
}

Then,%
\begin{equation}
\mathcal{G}(x)=\min\{\mathcal{M}(S(U))\mid\text{all\ open
subsets}\mathit{\ }U\subseteq X\text{ such that }x\in U\}.
\end{equation}

\begin{remark}\label{rerere}
Let us focus on the increasing muliti-family $\text{coGap}$ in $\mathbb{N}$ of example \ref{ex:Gap}.

Note, that for a set $S$ not to belong to $\text{coGap}$,
i.e.,\ to have $\text{coGap}(S)=0,$ just means that $S$ is finite - as a
`family, ignoring multiplicities' $\text{coGap}$ is just the family of
\textit{infinite} sets of natural\ numbers.

Thus, when we turn to the \textit{limit} of a sequence $(x_n)_{n\in\mathbb{N}}$
in a Hausdorff Space $X$ (a notion which is obviously dependent
on the topology. In a Banach or Hilbert space we will have strong and weak
limits etc.); and we take the $\text{coGap}$-limit (it will be a multi-set on
$X$, to which for some $x$ in $X$ to belong (at least) $n$ times, one must
have, for every neighborhood $U$ of $x$, that the $x_n$ stay in $U$ for some
$n$ consecutive places as far as we go); then the $\text{coGap}$-limit of
$(x_n)_{n\in\mathbb{N}}$, `forgetting the multiplicities' is just the set of
accumulation points of $(x_n)_{n\in\mathbb{N}}$ (which is, recalling the
examples \ref{exmps} in Section \ref{sbs:Ev}, just its $\mathcal{G}$-limit for
$\mathcal{G}$ the eventual family of the infinite subsets of $\mathbb{N}$).

Note that, in general, if the sequence has a limit $x^{\ast}$ (in the good old
sense) then its $\text{coGap}$-limit `includes $x^{\ast}$ infinitely many
times and does not include any other point'. This sort of indicates to what
extent the $\text{coGap}$-limit may be viewed as `more relaxed' than the
usual limit.

The inverse implication does not always hold (it holds however in a compact
space) as the following counterexample shows. In $\mathbb{R}$ (the reals),
define a sequence by%
\begin{equation}
x_{2n}:=n\text{ \ and \ }x_{2n-1}:=-1
\end{equation}
then its $\text{coGap}$-limit contains $-1$ infinitely often and does not
contain others, but $-1$ is not a limit.
\end{remark}

\subsection{Outer and Inner Increasing Multi-Families, The Outer Core and Inner Hull}\label{s:out-inn}
Let us try to relate a multi-family, a ($(0,1,\ldots,\infty$-valued) function on sets, with the likes of measures,
\begin{definition}
An increasing multi-family $\mathcal{M}$ on (the subsets) of a set $X$ will be called \textbf{outer} if it is a `finitely outer measure' on the subsets, i.e.\ if it satisfies, for any finite number of subsets $S_1,\ldots,S_k$
\begin{equation}\label{eq:out}
\mathcal{M}\left(\cup_{i=1}^k S_i\right)\le\sum_{i=1}^k\mathcal{M}(S_i).
\end{equation}
(for an infinite number of sets -- no way -- cf.\ the example of $\mathbb{N}$ as union of singletons.)
\end{definition}
Similarly,
\begin{definition}
An increasing multi-family $\mathcal{M}$ on (the subsets) of a set $X$ will be called \textbf{inner} if it is a `finitely inner measure' on the subsets, i.e.\ if it satisfies, for any finite number of \textbf{disjoint} subsets $S_1,\ldots,S_k$,
\begin{equation}\label{eq:inn}
\mathcal{M}\left(\cup_{i=1}^k S_i\right)\ge\sum_{i=1}^k\mathcal{M}(S_i).
\end{equation}
(which here easily implies the same for any infinite number of disjoint sets!)
\end{definition}

For a general increasing multi-family $\mathcal{M}$, its \textbf{outer core} $\text{Out }\mathcal{M}$
is defined as the biggest outer multi-family less than $\mathcal{M}$, namely, (as easily seen)
$$\text{Out }\mathcal{M}(S):=\min\left\{\sum_{i=1}^k\mathcal{M}(S_i)\,|\,\cup_{i=1}^k S_i=S\right\}.$$

And its \textbf{inner hull} $\text{Inn }\mathcal{M}$
is defined as the smallest inner multi-family greater than $\mathcal{M}$, namely, (as easily seen)
$$\text{Inn }\mathcal{M}(S):=\max\left\{\sum_{i=1}^k\mathcal{M}(S_i)\,|\,\cup_{i=1}^k S_i=S,\,
\text{the }S_i\text{ disjoint}\right\}.$$

Clearly $\text{Out }\mathcal{M}$ and $\text{Inn }\mathcal{M}$ thus defined will also be increasing.

\begin{remark}\label{rm:push-out-inn}
For a mapping $X\to Y$, the definition of a \textit{push} of a multi-family, and the fact that \textit{the inverse image of a complement, union, intersection,... is the complement, union, intersection,... of the inverse image(s)}, imply that \textbf{the push of an outer (resp.\ inner) multi-family is always outer (resp.\ inner)}.
\end{remark}

\medskip

\begin{remark}\label{rm:cogap}
As a little exercise, what about the example of the $(0,1)$-valued multi-family $\text{coGap}$ in $\mathbb{N}$ (example \ref{ex:Gap})?

As a brief reflection will show, \textit{it is not outer, rather its outer core is given, somewhat surprisingly, by}:

$$\text{Out }\,\text{coGap}(S):=\min(2,\text{coGap}(S)).$$

Indeed, any set $S\subset\mathbb{N}$ can be decomposed into two sets with $\text{coGap}\le 1$, that do not contain any intervals of 2 or more consecutive numbers -- take the subsets of the even (resp.\ odd) members of $S$.
Also, since $\text{coGap}$ is both finitely insensitive and vanishing only for finite sets, only decompositions into sets with non-zero value of $\text{coGap}$ matter. This makes an (infinite) set $S$ with $\text{coGap}=1$
have the same value also by $\text{Out}\,\text{coGap}$, and if $\text{coGap }(S)\ge2$ to have
$\text{Out}\,\text{coGap }(S)=2$.

By contrast, the \textit{inner hull} $\text{Inn}\,\text{coGap }(S)$ turns out `trivially', to be the mult-family equal \textit{infinity for any infinite subset of $\mathbb{N}$ while $0$ for a finite $S$}. That follows from the following `trick': every set $S$ with `intervals' of $n$ consecutive numbers as far as we go ($n>0$) we may partition into any number $K$ of \textbf{disjoint} sets with the same property (the $i$'th part as if picks the $mK+i$'th such intervals for $m=0,1,\ldots$.),
\end{remark}

\subsection{Image of a Multi-set by a Mapping}
We wish to consider what image a mapping can give to a multi-set on $X$ (such as a star of a multi-family, or a limit of a sequence with respect to a multi-family in $\mathbb{N}$, etc.)

\begin{definition}
Let $f:X\to Y$ be a mapping between sets, and let $L$ be a multi-set on $X$. Its multi-image $\textbf{multi-}f(L)$ is defined as the multi-set in $Y$:
$$\textbf{multi-}f(L)(y):=\sum_{x\in f^{-1}(y)}L(x).$$
(Note that there is no problem in the definition of the sum, since we are summing integers -- if, say, the sum comprises an infinite number of non-zero terms the sum is infinity.)
\end{definition}
Note that if $L$ takes only $\{0,1\}$-values, i.e.\ is (the characteristic function of) a set $S$, still $f(S)$ may be different from $\textbf{multi-}f(S)$, the latter may take values bigger than 1. Hence the special notation $\textbf{multi-}f(L)$.

\subsection{Some properties}

\begin{proposition}\label{prop}
\textbf{(i)a}. For a mapping between sets $f:X\to Y$, and an increasing multi-family $\mathcal{M}$ on (the subsets of) $X$,
$$\text{Out }(f\ast\mathcal{M})\ge f\ast(\text{Out }\mathcal{M}).$$

\medskip

\textbf{(i)b}. For a mapping between sets $f:X\to Y$, and an increasing multi-family $\mathcal{M}$ on (the subsets of) $X$,
$$\text{Inn }(f\ast\mathcal{M})\le f\ast(\text{Inn }\mathcal{M}).$$

\medskip

\textbf{(ii)}. For an increasing multi-family $\mathcal{M}$ on (the subsets of) a Hausdorff topological space $X$,
$$\text{Out}\,(\text{cl }\mathcal{M})\ge\text{cl}\,(\text{Out }\mathcal{M}).$$

\medskip

\textbf{(iii)}. For a continuous mapping between Hausdorff topological spaces $f:X\to Y$, and an increasing mult-family $\mathcal{M}$ on (the subsets of) $X$,
$$\text{cl }(f\ast\mathcal{M})\ge f\ast(\text{cl }\mathcal{M}).$$

\end{proposition}
\begin{proof}
\textbf{(i)a}: In computing $(\text{Out }(f\ast\mathcal{M}))(S)$, we minimize, for $S=\cup_{i=1}^k S_i$, on $\sum_{i=1}^k(f\ast\mathcal{M})(S_i)=\sum_{i=1}^k\mathcal{M}(f^{-1}(S_i))$. Now $\cup_{i=1}^k(f^{-1}(S_i))=
f^{-1}\left(\cup_{i=1}^k S_i\right)=f^{-1}(S)$. So here we are minimizing the sum of $\mathcal{M}$ on some ways of writing $f^{-1}(S)$ as a finite union, namely the union of $(f^{-1}(S_i))_{i=1}^k$. So we will get a $\le$ value if we minimize on \textbf{all} ways of writing $f^{-1}(S)$ as a finite union. But the latter gives $(\text{Out }\mathcal{M})(f^{-1}(S))=(f\ast(\text{Out }\mathcal{M}))(S)$.

\medskip

\textbf{(i)b}: In computing $(\text{Inn }(f\ast\mathcal{M}))(S)$, we maximize, for $S=\cup_{i=1}^k S_i$, $S_i$ disjoint, on $\sum_{i=1}^k(f\ast\mathcal{M})(S_i)=\sum_{i=1}^k\mathcal{M}(f^{-1}(S_i))$. Now $\cup_{i=1}^k(f^{-1}(S_i))=f^{-1}\left(\cup_{i=1}^k S_i\right)=f^{-1}(S)$. So here we are maximizing the sum of $\mathcal{M}$ on some ways of writing $f^{-1}(S)$ as a disjoint finite union, namely the union of $(f^{-1}(S_i))_{i=1}^k$. So we will get a $\ge$ value if we maximize on \textbf{all} ways of writing $f^{-1}(S)$ as a finite union. But the latter gives $(\text{Out }\mathcal{M})(f^{-1}(S))=(f\ast(\text{Out }\mathcal{M}))(S)$.

\medskip

\textbf{(ii)}: In computing $(\text{Out}\,(\text{cl }\mathcal{M}))(S)$ we minimize, for $S=\cup_{i=1}^k S_i$, on
$\sum_{i=1}^k((\text{cl }\mathcal{M})(S_i))$. Each $(\text{cl }\mathcal{M})(S_i)$ is the minimum of $\mathcal{M}(U_i)$ for all open $U_i\supset S_i$. So we are minimizing on $\sum_{i=1}^k\mathcal{M}(U_i)$ for all $S_i$ such that $S=\cup_{i=1}^k S_i$ and open $U_i\supset S_i$. But each $\sum_{i=1}^k\mathcal{M}(U_i))$\,\,$\ge$ than
$(\text{Out }\mathcal{M})(\cup_{i=1}^k U_i)$ and as $\cup_{i=1}^k U_i$ is open and contains $\cup_{i=1}^k S_i=S$,
that is $\ge$ than $(\text{cl}\,(\text{Out }\mathcal{M}))(S)$. Hence the assertion.

\medskip

\textbf{(iii)}: In computing $(\text{cl }(f\ast\mathcal{M}))(S)$ for sets $S\subset Y$, we minimize on $(f\ast\mathcal{M})(U)$ for open sets $U$ in $Y$ containing $S$. But $(f\ast\mathcal{M}))(U)=\mathcal{M}(f^{-1}U)$ and, $f$ being continuous, $f^{-1}U$ is open in $X$, and of course contains $f^{-1}S$. Therefore the latter is
$\ge(\text{cl }\mathcal{M})(f^{-1}S)=(f\ast(\text{cl }\mathcal{M}))(S)$. And we are done.
\end{proof}
\begin{remark}
Here we stated facts holding generally and with `straightforward' proofs. Maybe for special cases (such as compact topological spaces) these may be strengthened (say, equality in the inequalities) with more sophisticated arguments.
\end{remark}

Yet we mention one such case

\begin{proposition}\label{prop1}
\textbf{(ii)*}. For an increasing multi-family $\mathcal{M}$ on (the subsets of) a Hausdorff topological space $X$,
and a finite set $F$,
$$\text{Inn}\,(\text{cl }\mathcal{M})(F)\le(\text{cl}\,(\text{Inn}\mathcal{M}))(F).$$
\end{proposition}
\begin{proof}
In computing $(\text{Inn}\,(\text{cl}\mathcal{M}))(F)$ we maximize, for $F=\cup_{i=1}^k F_i$,\,\,$F_i$ disjoint (and necessarily finite), on $\sum_{i=1}^k((\text{cl }\mathcal{M})(F_i))$.

Pick a decomposition $F=\cup_{i=1}^k F_i$ which gives the maximum, that being the value of
$(\text{Inn}\,(\text{cl }\mathcal{M}))(F)$.

Each $(\text{cl }\mathcal{M})(F_i)$ is the minimum of $\mathcal{M}(U_i)$ for all open $U_i\supset F_i$, and, in a Hausdorff space, we may assume also the $U_i$ disjoint. Then we have $\sum_{i=1}^k\mathcal{M}(U_i)\le(\text{Inn }\mathcal{M})(\cup_{i=1}^k U_i)$, making (for that maximizing decomposition)
\begin{equation}\label{eq:pf}
(\text{Inn} (\text{cl}\mathcal{M}))(F)\le(\text{Inn }\mathcal{M})(\cup_{i=1}^k U_i).
\end{equation}
On the other hand, in computing $(\text{cl}\,(\text{Inn }\mathcal{M}))(F)$ we minimize on $(\text{Inn }\mathcal{M}))(U)$ for $U$ an open set containing (the finite) $F$, and, again since the space is Hausdorff and we speak about finite sets, for every decomposition $F=\cup_{i=1}^k F_i$, in particular for the above maximizing one, and every such $U$ there are disjoint open $U_i$, $U_i$ containing $F_i$, such that $\cup_{i=1}^k U_i\subset U$, implying, by equation (\ref{eq:pf}), that for each of our $U$'s\,\,
$(\text{Inn }\mathcal{M})(U)\ge(\text{Inn}\,(\text{cl }\mathcal{M}))(F)$.
\end{proof}

\subsection{Action of a Continuous Mapping on Limits with respect to \textbf{Inner} Multi-Families}
Thus, Propositions \ref{prop} and \ref{prop1} yield, but (the arguments I have work ...) only for \textbf{inner} multi-families,
\begin{theorem}\label{thm:lim}
Let $X$ and $Y$ be Hausdorff topological spaces, $f:X\to Y$ be continuous,
and $\mathcal{M}$ be an \textbf{inner} increasing multi-family on (the subsets of ) $X$. Then
$$\lim(f\ast\mathcal{M})\ge\textbf{multi-}f(\lim\mathcal{M})).$$
\end{theorem}
\begin{proof}
In Proposition \ref{prop} \textbf{(iii)}, take the star of both sides to get
$$\text{Star}\left(\text{cl}(f\ast\mathcal{M})\right)\ge
\text{Star}\left(f\ast(\text{cl }\mathcal{M})\right).$$
The LHS here is the (multi-set) $\lim(f\ast\mathcal{M})$. The RHS is the multi-set whose value at some $y\in Y$ is
\begin{eqnarray*}
&&\left(f\ast(\text{cl }\mathcal{M})\right)(\{y\})\\
&&=(\text{cl }\mathcal{M})(f^{-1}(y))\\
&&=(\text{cl}\,(\text{Inn }\mathcal{M}))(f^{-1}(y))
\end{eqnarray*}
($\mathcal{M}$ being inner). Then, by, Proposition \ref{prop1}, for every finite $F\subset f^{-1}(y)$
\begin{eqnarray*}
&&\ge(\text{Inn}\,(\text{cl }\mathcal{M}))(F)\\
&&=(\text{Inn}\,(\text{cl }\mathcal{M}))(\cup_{x\in F}\{x\})\\
&&\ge\sum_{x\in f^{-1}(y)}\,(\text{cl }\mathcal{M}))(\{x\})\\
&&=\sum_{x\in F}(\text{Star}\,(\text{cl }\mathcal{M}))(x)\\
&&=\sum_{x\in F}(\lim\mathcal{M})(x).
\end{eqnarray*}
That holding for any finite $F\subset f^{-1}(y)$ means
$$\ge\sum_{x\in f^{-1}(y)}(\lim\mathcal{M})(x)=(\textbf{multi-}f(\lim\mathcal{M}))(y).$$
\end{proof}
For sequences, noting remark \ref{rm:push-out-inn} we get
\begin{corollary}\label{Cor:inn-seq}
Let $X$ and $Y$ be Hausdorff topological spaces, $f:X\to Y$ be continuous,

and suppose $\mathcal{M}$ is an \textbf{inner} increasing multi-family in $\mathbb{N}$,

Then for a sequence $(x_n)$ in a $X$, (recall these limits are multi-sets on the spaces)
$$\mathcal{M}\text{-}\lim(f(x_n))\ge\textbf{multi-}f(\mathcal{M}\text{-}\lim x_n).$$
\end{corollary}

And compare what will be said in \S\ref{s:inn-ev}

\subsection{Special Properties of Inner Eventual Families}\label{s:inn-ev}

An inner eventual family, i.e.\ to which two disjoint sets cannot both belong (\textbf{(I)} in (\ref{it:inn})), shares with a filter (by the same reasoning) the fact that \textbf{in a Hausdorff topological space it can have at most one limit point, consequent for an inner eventual family $\mathcal{E}$ the $\mathcal{E}$-limit of a sequence in a Hausdorff space, if it exisrs, is unique.}

Thus, the statements about an inner $\mathcal{M}$, of Theorem \ref{thm:lim}:
$$\lim(f\ast\mathcal{M})\ge\textbf{multi-}f(\lim\mathcal{M})),$$
and of Corollary \ref{Cor:inn-seq}:
$$\mathcal{M}\text{-}\lim(f(x_n))\ge\textbf{multi-}f(\mathcal{M}\text{-}\lim x_n),$$
when restricted to an inner \textit{eventual family} (not general multi-families),
will really speak about a single limit point (if exists), and say that, \textit{as with limits in the (much more restrictive) good old sense},
\begin{theorem}
Let $X$ and $Y$ be Hausdorff topological spaces, $f:X\to Y$ be continuous.

\textbf{(i)} Let $\mathcal{F}$ is an \textbf{inner} eventual family on (the subsets of ) $X$.

Then, if $\lim\mathcal{F}$ exists (an element of $X$, necessarily unique),

then $\lim(f\ast\mathcal{M})$ also exists (an element of $Y$) and is equal to
$f(\lim\mathcal{M}))$.

\textbf{(ii)} Let $\mathcal{E}$ is an \textbf{inner} eventual family in $\mathbb{N}$.

Given a sequence $(x_n)$ in a $X$, if its $\mathcal{E}\text{-}\lim x_n$ in $X$ exists
(an element of $X$, necessarily unique),

then $\mathcal{E}\text{-}\lim(f(x_n))$ also exists and is equal to
$f(\mathcal{E}\text{-}\lim x_n)$
\end{theorem}

\section{More about Eventual Families, Associated and Self-Associated, Inner, Outer -- Generalization of Ultrafilters?}\label{s:Ev-Sa-Inn-Out}

\subsection{Dually Associating an Eventual Family $\mathcal{F}$ with the Family of Sets with Complements not in $\mathcal{F}$}\label{s:aso}
For an eventual family $\mathcal{F}$ in some set $X$, Proposition \ref{lem:simple-observ} told us that its complement -- the family of sets not in $\mathcal{F}$, is a co-eventual family and vice-versa. But clearly also the \textit{set of complements of the members of $\mathcal{F}$} is a co-eventual family and vice-versa.

These are different. For instance, for the family $\mathcal{G}$ of all infinite subsets of $\mathbb{N}$ and the family $\mathcal{H}$ of all subsets of $\mathbb{N}$ with finite complement (cf.\ examples \ref{exmps}), the complement of $\mathcal{G}$ is the family of finite sets -- indeed just the set of complements of the members of $\mathcal{H}$, while the family of complements of the members of $\mathcal{G}$ -- all subsets of $\mathbb{N}$ with infinite complements -- is just the complement of $\mathcal{H}$!

This may be generalized. For any family $\mathcal{F}$ in some set $X$, as if `go to the co-eventual one way and return the other way' to get what we call here its \textbf{associate} eventual family $\mathcal{F}$ in $X$, defined as
$$\text{Aso }(\mathcal{F}):= \text{the family of all subsets whose complement is not in}\,\mathcal{F}.$$
Clearly $\text{Aso}\circ\text{Aso}=\text{id}$ -- it is an involution, a kind of duality. And as we saw
it pairs, in particular, the above $\mathcal{G}$ and $\mathcal{H}$ on (the subsets of) $\mathbb{N}$.%

Note that \textbf{The `larger' $\mathcal{F}$ is the `smaller' $\text{Aso }\mathcal{F}$ will be}.

For a mapping $X\to Y$, the definition of a \textit{push} of an eventual family, and the fact that \textit{the inverse image of a complement is the complement of the inverse image}, imply that \textbf{the associate of the push is the push of the associate}.

\medskip

\subsection{Partial Generalization to Multi-Families}
For a multi-family on (the subsets) of a set $X$, one readily generalizes \textit{taking the family of complements} (cf.\ example \ref{ex:Gap}) to $\mathcal{M}\hookrightarrow \text{co }\mathcal{M}$ defined by
$\text{co }\mathcal{M}(S):=\mathcal{M}(S^c)$\,\,($S^c$ denotes the complement of $S$).

The other way -- `the complement of $\mathcal{F}$, i.e.\ the family of sets not in $\mathcal(E)$' does not seem to generalize readily, so we do not have a duality such as $\text{Aso}$.

Yet we can easily see that

\begin{proposition}
If $\mathcal{M}$ is an \textit{increasing} multi-family on (the subsets of) some set $X$, then the associate of the (eventual family) the (upper) level set $\{S\,|\,\mathcal{M}(S)\ge n+1\}$ is the (eventual family) the (lower) level set (of, here, a \textit{decreasing} multi-family) $\{S\,|\,(\text{co }\mathcal{M})(S)\le n\}$. (recall, these are integers, $\le n$ is the negation of $\ge n+1$ -- one need not bother with $<$ etc.)
\end{proposition}
\begin{proof}
\end{proof}

\subsection{Inner and Outer}
Let us focus on an eventual family $\mathcal{F}$ on (the subsets of) a set $X$, identified with its \textit{characteristic functions}, to be thought of as just a $(0,1)$-valued increasing multi-family.

And one may ask: when are these outer, inner, `finitely additive'? Also, recalling \S\ref{s:aso}, when such an eventual families is \textit{the associate of itself?}.

These are related to viewing the above characteristic function of $\mathcal{F}$, as in \S\ref{s:out-inn},
as ($(0,1)$- valued) likes of measures.

\begin{enumerate}
\item\label{it} Let us clarify in what way the condition for being a (finitely-additive) `measure'; `outer (resp.\ inner) measure' applies here. That concerns two \textbf{disjoint} subsets $S$ and $T$, comparing the values the characteristic function of $\mathcal{F}$ gives to $S$,\,$T$ and $S\cup T$, i.e.\ do these belong to $\mathcal{F}$ or not.

    When one of $S$ and $T$ has `measure' $0$ (i.e.\ is not in $\mathcal{F}$) while the other has `measure' $1$ (i.e.\ does belong to $\mathcal{F}$), there is only one possibility: since $\mathcal{F}$ is eventual, the union $S\cup T$ also belongs to $\mathcal{F}$, has thus `measure' $1$, and there is always `full' finite additivity: $0+1=1$.

    So we need to focus on the cases where both $S$ and $T$ have measure $0$ or both have measure $1$.

\item When both $S$ and $T$ have `measure' $0$, i.e.\ are not in $\mathcal{F}$, the union, $S\cup T$, of course, has `measure' $0$ or $1$. So \textit{the condition (\ref{eq:inn}) for `inner' always holds here}, while for equality, i.e.\ finite additivity, equivalently \textit{the condition (\ref{eq:out}) for `outer'} here, we need that also $S\cup T$ has `measure' $0$, i.e.\ is not in $\mathcal{F}$. Thus for finite additivity, equivalently `outer', to always hold here we need that

{\flushleft\textbf{(O)}\qquad\textbf{the union of two sets (hence of a finite number of sets) not in $\mathcal{F}$ is also not in $\mathcal{F}$}}

\item Analogously, when both $S$ and $T$ (hence, of course, also $S\cup T$) have `measure' $1$, i.e.\ belong to $\mathcal{F}$, \textit{the condition (\ref{eq:out}) for `outer' will always hold here}, we will never have equality, i.e.\ finite additivity, and also \textit{the condition (\ref{eq:inn}) for `inner'} never holds here.

\item\label{it:inn} Therefore when $\mathcal{F}$ is inner, in particular `finitely additive', the last case cannot occur, that is

{\flushleft\textbf{(I)}\qquad\textbf{two disjoint subsets cannot both belong to $\mathcal{F}$}}

\item\label{it:flt} So a filter is inner (unless it is the `trivial' family of all subsets) -- if it has disjoint members, it will contain their intersection $\emptyset$.

\item And we may conclude: $\mathcal{F}$ is outer if and only if \textbf{(O)} always holds, is inner if and only if \textbf{(I)} always holds, and gives a `finitely additive measure' if and only if both always hold.

\item The condition for outer here, Condition \textbf{(O)}, may be phrased thus: A set does not belong to $\mathcal{F}$ iff its complement belongs to $\text{Aso }\mathcal{F}$ (see \S\ref{s:aso}). But (de Morgan laws) the union of complements is the complement of the intersection. So \textbf{(O)} may be translated into: in $\text{Aso }\mathcal{F}$ the intersection of members is also a member, i.e.\ $\text{Aso }\mathcal{F}$ is a \textbf{filter}. Hence the condition \textbf{(O)} for outer is equivalent to

{\flushleft\textbf{(O${}^*$)}\qquad\textbf{the associate $\text{Aso }\mathcal{F}$ is a filter}}

\item Now, which eventual families are \textit{the associate of themselves?} For that $\mathcal{F}$ must have the property that a set $S$ belongs to $\mathcal{F}$ iff its complement $S^c$ does not.

    That implies, yet is stronger than, the `inner' condition \textbf{(I)}.

    If $\mathcal{F}$ is both inner and outer (and is not the trivial empty family), thus furnishes a `finitely additive measure', it clearly will be self-Aso.

    Indeed, the `measures' of $S$ and $S^c$, each $0$ or $1$, must add up to the measure of the whole set $X$ which is $1$ since $X\in\mathcal{F}$. Hence each is $0$ iff the other is $1$.

\item But a self-Aso eventual family \textbf{may not be `finitely additive'}, and we shall see (Proposition \ref{prop:flt}) that it is `finitely additive' iff it is moreover a filter.
\end{enumerate}

\begin{example}\label{ex:odd}
A simple example (which will have some general significance) of a self-Aso eventual family $\mathcal{F}$ which is not finitely additive and not a filter.

In a \textit{finite set $X$ with odd number $2N+1>1$ of elements}, let $\mathcal{F}$ is the family of all sets with $\ge N+1$ elements, i.e.\ which include more than half the elements of $X$.

It (i.e.\ its characteristic function) is not a `finitely additive probability measure' -- the singletons have measure $0$ and their union -- the whole $X$ has measure $1$.

And it clearly is not a filter -- the intersection of sets with $>N$ elements can very well have $\le N$ elements.

As we shall see, one may use this example to construct similar examples in an infinite $X$.

In particular, we may transform that to an example in an infinite set $X$: pick a finite subset $F\subset X$ with
$2N+1>1$ elements and define $\mathcal{F}$ on $X$ by: $\mathcal{F}$ is the family of all subsets of $X$ \textit{whose intersection with $F$ has $\ge N+1$ elements}.
\end{example}

Also, in the example of the pair $\mathcal{G}$ and $\mathcal{F}$, the latter is a filter, so $\mathcal{G}$ is outer.
About $\mathcal{F}$ -- the family of sets with finite complement -- the same trick as in Remark \ref{rm:cogap} of decomposing a set to its even (resp.\ odd) elements \textrm{which both will not belong to $\mathcal{F}$}, shows that $\text{Out }\mathcal{F}$ is \textrm{the zero multi-family, identified with the empty family}.

This shows that one should not `dismiss' the two `trivial' families: the empty and the one comprising all subsets --
they may have significant roles!

\subsection{A `Litmus Paper': Measured and non-Measured Partitions}
\begin{definition}\label{def:mea}
Let $\mathcal{F}$ be a self-Aso eventual family on (the subsets) of a set $X$.

Let $X=X_1\cup X_2\cup\ldots X_k$ (the $X_i$ disjoint) be a finite partition of $X$.

If some $X_i$ is in $\mathcal{F}$ the partition is called \textbf{measured}. As $\mathcal{F}$ is a self-Aso eventual family, the complement of $X_i$ is then not in $\mathcal{F}$. As the latter contains all the unions of the $X_j, j\neq i$, also these are not in $\mathcal{F}$. While the other finite unions -- those including $X_j$ must be in $\mathcal{F}$ since $X_i$ is.

If no $X_i$ is in $\mathcal{F}$ call the partition \textbf{non-measured}. Then $\mathcal{F}$ cannot be a `finitely additive measure' since the $X_i$ have `measure' $0$ and their union $X$ has `measure' $1$
\end{definition}
Note that when the partition is into \textbf{three} parts, the two-element unions are the complements of the parts, so the value the `measure' gives to the parts determines everything concerning the unions of parts.

\begin{proposition}\label{prop:flt}
Let $\mathcal{F}$ be a self-Aso eventual family on (the subsets) of a set $X$. TFAE

\begin{enumerate}
\item\label{1} (The characteristic function) of $\mathcal{F}$ is finitely additive

\item\label{2} There are no non-measured partitions

\item\label{3} There are no non-measured partitions into three sets.

\item\label{4} $\mathcal{F}$ is a filter (thus an ultrafilter)
\end{enumerate}
\end{proposition}
\begin{proof}
That \ref{1}$\Rightarrow$\ref{2} we have noted above: The existence of a non-measured partition, with parts of `measure' $0$ and their (finite) union $X$ of `measure' $1$, clearly violates finite additivity.

\ref{3}$\Rightarrow$\ref{1}: Let $S$ and $T$ be disjoint. Then we have the three-part partition $S,T,(S\cup T)^c=S^c\cap T^c$. The partition is measured, therefore one of the three parts have `measure' $1$ and the others
`measure' $0$. As in Definition \ref{def:mea} this determines the values for all partial unions of the parts and means that `finite additivity' holds whenever these are concerned, in particular for $S$, $T$ and their union.

\ref{2}$\Rightarrow$\ref{3} is obvious.

\ref{3}$\Rightarrow$\ref{2}: Firstly, to a non-measured partition with $2$ parts, just add $\emptyset$ as a part.

We show that if there is a non-measured partition into $K>3$ parts, one can find one with $<K$ parts. Indeed, unite two parts $S$ and $S'$ to get a partitions into $K-1$ parts. If that union of two parts is not in $\mathcal{F}$, the new partition is non-measured with $K-1$ parts. If it is in $\mathcal{F}$, its complement, i.e.\ the union of all parts other than $S$ and $S'$, is not in $\mathcal{F}$, so the latter complement, together with $S$ and $S'$, make a non-measured with $3$ parts.

\ref{3}$\Rightarrow$\ref{4}. For $\mathcal{F}$ to be a filter, it must contain $S\cap T$ if $S,T\in\mathcal{F}$. Now, we have a partition of $X$ into three sets, $S\cap T, S^c, S\cap T^c$. of these, the union of $S\cap T$ and $S^c$ includes $T$, so is in $\mathcal{F}$. And the union of $S\cap T$ and $S\cap T^c$ includes $S$, so is in $\mathcal{F}$ too. Their respective complements $S\cap T^c$ and $S^c$ thus cannot be in $\mathcal{F}$. So, for the partition to be measured $S\cap T$ must be in $\mathcal{F}$.

\ref{4}$\Rightarrow$\ref{1}: Assume $\mathcal{F}$ is a filter. For finite additivity we need that if $S$ and $T$ are disjoint, the `measure' of $S\cup T$ is the sum of the `measures' of $S$ and $T$.

As we saw in (\ref{it}) above, if one `measure' is $0$ and the other is $1$ thing are always OK.

Also, a filter being inner, the measures cannot be both $1$ -- see (\ref{it:inn}) and (\ref{it:flt}) above.

If these both `measures' are $0$, i.e.\ both $S$ and $T$ are not in $\mathcal{F}$, then $\mathcal{F}$ includes their complements $S^c$ and $T^c$, hence, being a filter, includes $S^c\cap T^c=(S\cup T)^c$ (De Morgan's laws), so $S\cup T$ is not in $\mathcal{F}$ -- has `measure' $0$.
\end{proof}

\subsection{Ultrafilters vs.\ Inner; Outer; Self-Aso Eventual Families}\label{sbs:u}
So, we have inner eventual families, including the filters and self-Aso, which include the self-Aso filters, equivalently those furnishing a `finitely additive measure'

The latter may be characterized as `self-Aso filters' -- filters to which a set belongs iff its complement does'nt.  Alternatively as \textit{maximal filters} -- the lack of a set $S$ witn both $S$ and $S^c$ not in $\mathcal{U}$ is exactly what hampers extending to a larger filter. These are the ultrafilters.

As mentioned in the end of \S\ref{sbs:u}, self-Aso eventual families share part of the properties of ultrafilters:

They both `extract' a $(0,1)$ value -- yes/no if you wish -- from such value at each member of a set, and consequently   (see item \ref{it:lim} below) may `extract' a numerical value -- a limit, from numerical values at each member (i.e.\ a numerical function).

Self-Aso eventual families, like ultrafilters, would respect inclusion/implication and complement/negation, but not the binary Boolean `and' and `or'. Thus they would not do for most of the roles ultrafilters have in Logic, etc., but maybe would still provide a generalization for some.

Note also that

\begin{enumerate}
\item In a finite set every ultrafilter is `fixed' at some element $x$ -- is the family of subsets containing $x$. Thus to a finite set they will not add any `limit points' in the sense of \S\ref{sbs:u}. As our `majority' example \ref{ex:odd} shows, that is certainly not the case with self-Aso.

\item \label{it:frgt} With a self-Aso family $\mathcal{F}$, in particular an ultrafilter, by the way things (limits, integrals) are defined, if $Q$ is a subset which belongs to $\mathcal{F}$, equivalently $Q^c$ does belongs there, then any change in $Q^c$ \textbf{does not matter at all} -- indeed it is a set of measure $0$ here -- and one may view $\mathcal{F}$ as a self-Aso family in $Q$, `and forget about $Q^c$'. Conversely, given a self-Aso family $\mathcal{F'}$ on a $Q\subset X$ one `equates' it with the family on $X$\,\,$\mathcal{F}:=\{S\subset X\,|\,S\cap Q\in\mathcal{F'}\}.$

\item Yet inner eventual families (including the self-Aso) do not share with filters the fact that \textit{In a (Hausdorff) compact space, they must have some limit points, and a unique limit point must be a limit} - for that one need that every finite partition be measured.

    \item \label{it:lim} In any (Hausdorff) \textit{compact} space an ultrafilter $\mathcal{U}$ has a unique limit. This limit may be viewed also as an \textit{integral}. The limit is an $x_0\in X$ all whose neighborhoods belong to $\mathcal{U}$. Thus their complements are not in $\mathcal{U}$, so `can be completely forgotten' according to item \ref{it:frgt}. It is as if `everything is' in any neighborhood of $x_0$, which, of course, can be taken as small as we wish).

    This hallmark ultrafilters, seems to be largely shared by self-Aso families. (Just, contrary to ultrafilters, one may not have a limit).
\end{enumerate}

\subsection{Constructing new Self-Aso from Old}
Similarly to ultrafilters, one may construct new self-Aso from old (by sort of `concatenating' the, if you wish, `extracting a $(0,1)$ value -- yes/no, from a family of such'). When one of the ingredients is self-Aso but not an ultrafilter (as in example \ref{ex:odd}) one would get many examples of such.

Of course, any function $X\to Y$ would \textbf{push} ultrafilters (resp.\ self-Aso) in $X$ to such in $Y$.

Another example is a \textbf{product}: Let $X$ and $Y$ be sets, $\mathcal{E}$ and $\mathcal{F}$ self-Aso eventual families  on (the subsets of) $X$ and $Y$ respectively.

One may `straightforwardly' define such on the Cartesian product $X\times Y$ as follows:

To check whether a subset $S$ of $X\times Y$ (i.e.\ -- its characteristic function -- a $(0,1)$-valued function on $X\times Y$) belongs, one just `integrates' on the $X$-slices using $\mathcal{E}$ -- `extracts' a $(0,1)$-value for each slice, and then `integrates' these on $Y$ using $\mathcal{F}$.

The problem is, that as with ultrafilters, \textbf{the order $Y$ after $X$ or $X$ after $Y$ matters very much} -- no analog of Fubini's theorem.

While, as said, one does get many examples, That `asymmetry' with product hampers `fully' using ultrafilters, let alone self-Aso to `extend sets by limits'. One needs to extend `non-canonically' to get what is called Non-standard Analysis (see references in the bibliography, say my \cite{Levy}).

One tries to modify that product scenario:

Let $X$ be a set, let us be given a self-Aso eventual family $\mathcal{F}$ on $X^{\{1,2,\ldots,2N+1\}}$ -- the $2N+1$-tuples in $X$ (or, see item \ref{it:frgt} above, -- on some subset $Q$ therein), then one `projects' to a self-Aso eventual family in $X$ by `deciding whether a subset $S\subset X$ belongs' by considering \textit{the set of $2N+1$-tuples with `majority' lying in $S$}, and seeing whether that belongs to $\mathcal{F}$.

(At the end of example \ref{ex:odd} we, in fact, did that with $\mathcal{F}$ a fixed ultrafilter)

\textbf{Acknowledgment}. This work is an offshoot of a collaboration with Prof.\ Yair Censor of the
University of Haifa (see \cite{Censor-Levy}, \cite{Censor-Levy-2019}). I am very much indebted to Professor Censor for raising the questions and many helpful discussions and advice.

\end{document}